\documentclass[11pt, twoside]{article}

\usepackage{graphicx}
\usepackage[caption=false]{subfig}
\usepackage{amsmath}
\usepackage{amssymb,amsfonts}
\usepackage{amsthm}
\usepackage{bm}
\usepackage{mathrsfs}
\usepackage{amssymb}
\usepackage{multirow}
\usepackage[most]{tcolorbox}
\usepackage[margin=1in]{geometry}
\usepackage{algorithm}
\usepackage{algpseudocode}
\numberwithin{equation}{section}
\usepackage{multirow}
\usepackage{makecell}
\usepackage{booktabs}
\theoremstyle{definition}
\newtheorem{theorem}{Theorem}

\newtheorem{lemma}{Lemma}

\newtheorem{remark}{Remark}
\newtheorem{example}{Example}
\newtheorem{assumption}{Assumption}

\usepackage{cite}
\usepackage{hyperref}
\usepackage[nameinlink]{cleveref}
\newcommand{\vertiii}[1]{{\left\vert\kern-0.25ex\left\vert\kern-0.25ex\left\vert #1 
		\right\vert\kern-0.25ex\right\vert\kern-0.25ex\right\vert}}

\usepackage{fancyhdr}
\begin{document}
	
	\title{A new reduced order model of imcompressible Stokes equations}

\author{Yangwen Zhang
	\thanks{Department of Mathematics Science, Carnegie Mellon University, Pittsburgh, PA, USA (\mbox{yangwenz@andrew.cmu.edu}). Y.\ Zhang is supported by the US National Science Foundation (NSF) under grant number  DMS-2111315.}
}

\date{\today}

\maketitle

\begin{abstract}
		In this paper we propose a new reduced order model (ROM) to the imcompressible Stokes equations. Numerical experiments show that our ROM is accurate and efficient. Under some assumptions on the problem data, we prove that the convergence rates of the new ROM is the same with standard solvers.
\end{abstract}

\section{Introduction}
	Let $\Omega \subset \mathbb{R}^{d}, d=2,3$, be a regular open domain with Lipschitz continuous boundary $\partial\Omega$. We consider the  following imcompressible Stokes equation  with no-slip boundary conditions:
\begin{subequations}\label{Strong_form_Stokes}
	\begin{align}
	u_{t} -\nu \Delta u+\nabla p &=f \qquad\qquad\qquad\qquad\qquad \textup { in } \Omega \times(0, T],\label{moment}\\
	\nabla \cdot u&=0 \qquad\qquad\qquad\qquad\qquad\textup { in } \Omega \times(0, T], \\
	u&=0 \qquad\qquad\qquad\qquad\qquad\textup { on } \partial\Omega \times(0, T],\\
	u(x, 0)&=u_{0}(x) \qquad\qquad\qquad\qquad\textup { in } \Omega,		
	\end{align}		
\end{subequations}
where $u$ is the velocity, $p$ is the pressure, $f$ is the known body force, and $\nu$ is the viscosity.

In recent years, model order reduction (MOR) becomes more popular for solving partial differential equations (PDEs); see \cite{Singler_New_SINUM_2014,MR4172732,Rathinam_Petzold_New_SINUM_2003,Mancinelli_Pagliaroli_Camussi_Castelain_Hydrodynamic_Fluid_Mech_2018,Ali_Cortina_Hamilton_Calaf_Cal_Turbulence_Fluid_Mech_2017,Resseguier_Valentin_Etienne_Heitz_Chapron_Flow_Fluid_Mech_2017,Siebe_Paschereitr_Oberleithner_Spectral_Fluid_Mech_2016,MR3894167,MR4296762}. Especially, 
there has been a growing interest in the application of ROMs to modeling incompressible flows \cite{MR4133479,MR3825393}. These ROMs  use experimental data, or solutions generated from full order model (FOM); however,  when data is changed, there is no guarantee that the solution of the ROMs is accurate. 	In \cite{WalingtonWeberZhang}, we proposed a new  ROM for the heat equation with changing data.  We showed that the convergence rate of  the new ROM is the same with the standard solvers, and the  computational cost of the new ROM  can be orders of magnitude smaller when compared to these full-order schemes. 

Hence it is nature to ask, can we  extend the idea in \cite{WalingtonWeberZhang} to the Stokes equation?

In this paper we give a positive answer to the above question. First, we generate a sequence by solving a small number of \emph{time independent} Stokes equations; see \eqref{Generate_basis}. Second, we solve two optimization problems to get reduced velocity and pressure spaces; see \eqref{P1} and \eqref{P2}. The last step is to project the Stokes equation onto the reduced subspace, and obtain a velocity-only ROM; see \eqref{Velocity_onlyROM}. This is due to the fact that the reduced velocity space is weakly divergence free, and hence the pressure term was dropped out of the ROM formulation.  To recover the pressure, we use the momentum equation recovery approach, which was proposed in \cite{MR4133479,MR3894167}. The numerical experiment in \Cref{Example3} shows that the new ROM is accurate and efficient. Furthermore, in \Cref{Main_res,error_for_p} we prove that the convergence rates of the new ROM is the same with the standard solvers.

\section{Notation and preliminaries}
Throughout the paper, we assume $\Omega\subset \mathbb R^d$, $d=2,3$ is a bounded polyhedral domain. In order to give the weak form of the Stokes system \eqref{Strong_form_Stokes} we need to introduce two function spaces $V$ and $Q$ for the velocity $u$ and pressure $p$, respectively. Let
\begin{align*}
V &:= [H_{0}^{1}(\Omega)]^{d}=\left\{v \in [H^{1}(\Omega)]^{d}:\left.v\right|_{\partial\Omega}=0\right\}, \\
Q &:=L_{0}^{2}(\Omega)=\left\{q \in L^{2}(\Omega): \int_{\Omega} q d x=0\right\} .		
\end{align*}

We denote by $\|\cdot\|$ the $L^{2}(\Omega)$ norm and by $(\cdot, \cdot)$ the inner product, $\|\cdot\|_V$ the $V$ norm and by $(\cdot, \cdot)_V$ the inner product. For functions $v \in V$, the Poincar\'e inequality holds
\begin{align}\label{Poincare}
\|v\| \leq C_{P}\| v\|_V.	
\end{align}

Let $V'$ be the dual space of bounded linear functionals defined on $V$, and let  $\langle\cdot, \cdot\rangle_{V^{\prime}, V}$ denotes the duality pairing between $V$ and its dual $V'$. Then the space $V'$ is equipped with the norm
\begin{align*}
\|f\|_{V'}=\sup _{0 \neq v \in V} \frac{\langle f, v\rangle_{V',V}}{\| v\|_V} \quad \forall f \in V.		
\end{align*}

Now we can define the weak solution of \eqref{Strong_form_Stokes}: find $(u,p)\in V\times Q$ satisfying
\begin{align}\label{Weak_Stokes}
\begin{split}
\left(u_{t}, v\right)+a( u,  v)+b(v,p) &=(f, v)  \qquad\qquad\qquad\qquad\forall v \in V, \\
b(u,q) &=0 \qquad\qquad\qquad\qquad\qquad\forall q \in Q,	
\end{split}	
\end{align}
where
\begin{align*}
a(u,v) = \nu (\nabla u, \nabla v), \quad b(u,q) = -(\nabla\cdot u, q).
\end{align*}

For the above problem we know that the following inf-sup condition holds: there exists $\beta>0$ such that 
\begin{align}\label{Continuous_inf_Sup}
\beta \leq \inf _{q \in Q} \sup _{v \in V} \frac{|b(v, q)|}{\|v\|_{V}\|q\|_Q}.
\end{align}
For any $(u, p)$, $(v, q) \in V \times Q$ we define
\begin{align*}
\mathscr A((u, p),(v, q)) = a(u, v)+b(v, p)+b(u, q), \quad \textup{and}\;\;\|(v, q)\|_{V \times Q}^{2}=\|v\|_{V}^{2}+\|q\|_Q^{2}.
\end{align*}
\begin{lemma}\label{A_Stabilitya}
	There exists a constant $C>0$ such that
	\begin{align}\label{A_Stability}
	\sup_{(v, q) \in V \times Q} \frac{\mathscr A ((u, p),(v, q))}{\|(v, q)\|_{V \times Q}} \geq C\|(u, p)\|_{V \times Q}.	
	\end{align}
\end{lemma}	

In order to formulate a numerical method,  let $V_{h}$ and $Q_{h}$ be two spaces of piecewise polynomials that approximates $V$ and $Q$. Furthermore, we assume that the two finite element spaces $V_h$ and $Q_h$ satisfy the so-called discrete inf-sup condition: there exists $\beta_h>0$ such that
\begin{align}\label{Discrtr_inf_Sup}
\beta_h \leq \inf _{q_h \in Q_h} \sup _{v_h \in V_h} \frac{|b(v_h, q_h)|}{\|v_h\|_{V}\|q_h\|_Q}.
\end{align}
Then the stability in \Cref{A_Stabilitya} also holds on $V_h\times Q_h$. 

The condition \eqref{Discrtr_inf_Sup} is satisfied by several mixed finite elements, e.g., the Taylor-Hood elements and the MINI elements. In this paper, we shall use the Taylor-Hood elements for the numerical analysis and numerical experiments.

To simplify the presentation, we assume the initial condition $u_0 =0$ and the source term $f$ does not depend on time. The semidiscrete  finite element approximation of \eqref{Weak_Stokes} takes the form: find $u_{h}(t) \in V_{h}$ with $u_h(0)=0$,  and $p_{h}(t) \in Q_{h}$ such that
\begin{align}\label{Weak_Stokes_FEM}
\begin{split}
\left(\frac{{\rm d}}{{\rm d} t}u_h, v_h\right)+a( u_h,  v_h)+b(v_h,p_h) &=(f, v_h)  \qquad\qquad\qquad\qquad\forall v_h \in V_h, \\
b(u_h,q_h) &=0 \qquad\qquad\qquad\qquad\qquad\;\forall q_h \in Q_h.
\end{split}	
\end{align}

Next, we consider a discretization of the time interval $[0, T]$ into $N_T$ separate intervals such that $\Delta t=\frac{T}{N_T}$ and $t_{n}=n \Delta t$ for $n=0, \ldots, N_T$. We apply the backward Euler for the first step and then apply the two-steps backward differentiation formula (BDF2) for the time discretization.  Specifically, given $u_{h}^{0} =0$, we find $u_h^n\in V_h$ and $p_h^n\in Q_h$ satisfying
\begin{subequations}\label{CN_Stoeks}
	\begin{align}
	\left(\partial_t^+ u_h^n, v_{h}\right)+  a\left( u_h^n,  v_{h}\right) 
	+b\left(v_{h}, p_{h}^{n}\right)&=\left(f, v_{h}\right)  \qquad\;\; \forall v_{h} \in V_{h}, \label{BEStokes1}	\\
	b\left(u_{h}^{n}, q_{h}\right)&=0  \qquad \qquad\quad \forall q_{h} \in Q_{h},	\label{BEStokes2}
	\end{align}		
\end{subequations}
where
\begin{align*}
\partial_t^+ u_h^n = \begin{cases}
\dfrac{u_h^n - u_h^{n-1}}{\Delta t},\qquad\qquad\quad\;  n=1,\\[0.4cm]
\dfrac{3u_h^n - 4u_h^{n-1} + u_h^{n-2}}{2\Delta t},\quad n\ge 2.
\end{cases}
\end{align*}	

The computation can be extremely expensive if the mesh size $h$ and time step $\Delta t$ are small. 
\begin{example}\label{Example0}
	Let  $\Omega = (0,1)\times (0,1)$ and $T=1$, we consider the  Stokes equation \eqref{Strong_form_Stokes} with 
	\begin{align*}
	u_0 =  \left[\begin{array}{cccc}
	0\\[0.2cm]
	0\\
	\end{array}\right],\quad 	f =  \left[\begin{array}{cccc}
	f_1\\[0.2cm]
	f_2\\
	\end{array}\right], \quad f_1 = 100\sin(x)\exp(x), \quad f_2 = 100\cos(x)\exp(y).
	\end{align*}
	
	We use \texttt{$P_2-P_1$} Taylor-Hood element for spatial discretization and BDF2 for time discretization with time step $\Delta t = h^{3/2}$, here $h$ is	the mesh size (max diameter of the triangles in the mesh). We report the wall time\footnote{All the code for all examples in the paper has been made by the author using MATLAB R2021a
		and has been run on a laptop with MacBook Pro, 2.3 Ghz8-Core Intel Core i9 with 64GB 2667 Mhz DDR4. We use the Matlab built-in function \texttt{tic}-\texttt{toc} to denote  the wall time. }  in \Cref{table_0}. 
	\begin{table}[H]
		\centering
		{
			\begin{tabular}{c|c|c|c|c|c|c|c}

				\cline{1-8}
				$h$	& $1/2^1$ &$1/2^2$ 
				&$1/2^3$ 
				&$1/2^4$  &$1/2^5$ &$1/2^6$ 
				&$1/2^7$  
				\\
				\cline{1-8}
				{Wall time}
				&0.14&0.04&0.16& 	  0.92&   13.8&   216&   
				3851 \\ 
				\Xhline{0.01pt}

			\end{tabular}
		}
		\caption{ The wall time (seconds) for the simulation of \Cref{Example0}.}\label{table_0}
	\end{table}		
\end{example}


\section{Reduced order model (ROM)}\label{VOR}
For a given integer $\ell$ (small) and let $\mathfrak{u}_h^0 = f$.  For $1\le i \le \ell$,  we find $\left(\mathfrak{u}_h^i, \mathfrak{p}_h^i\right)\in V_h \times Q_h$ satisfying
\begin{align}\label{Generate_basis}
\begin{split}
a\left( \mathfrak{u}_h^i,  v_{h}\right) 
+b\left(v_{h}, \mathfrak{p}_h^i\right) &=\left(\mathfrak{u}_h^{i-1}, v_{h}\right) \qquad \qquad \qquad  \forall v_{h} \in V_{h}, \\
b\left( \mathfrak{u}_h^i, q_{h}\right)&=0  \qquad \qquad \qquad\qquad \qquad  \forall q_{h} \in Q_{h}.
\end{split}
\end{align}
Then, we consider the following minimization problems
\begin{gather}\tag{P1}\label{P1}
\begin{split}
\min _{{\widetilde \varphi}_{1}, \ldots, {\widetilde\varphi}_{r_u} \in V_h}  \sum_{j=1}^{\ell}  \left\|\mathfrak{u}_h^j-\sum_{i=1}^{r_u}\left( \mathfrak{u}_h^j, {\widetilde\varphi}_{i}\right)_{V} {\widetilde\varphi}_{i}\right\|_{V}^{2} \quad \text { s.t. }\left({\widetilde\varphi}_{i}, {\widetilde\varphi}_{j}\right)_V=\delta_{i j}, 1 \leq i, j \leq r_u,	
\end{split}
\end{gather}
and
\begin{gather}\tag{P2}\label{P2}
\begin{split}
\min _{{\widetilde \psi}_{1}, \ldots, {\widetilde\psi}_{r_p} \in Q_h}  \sum_{j=1}^{\ell}  \left\|\mathfrak{p}_h^j-\sum_{i=1}^{r_p}\left( \mathfrak{p}_h^j, {\widetilde\psi}_{i}\right) {\widetilde\psi}_{i}\right\|^{2} \quad \text { s.t. }\left({\widetilde\psi}_{i}, {\widetilde\psi}_{j}\right)=\delta_{i j}, 1 \leq i, j \leq r_p. 	
\end{split}
\end{gather}	

%


Let $\{\widetilde \varphi_1, \widetilde\varphi_2,\ldots, \widetilde\varphi_{r_u}\}$ and $\{\widetilde \psi_1, \widetilde\psi_2,\ldots, \widetilde\psi_{r_p}\}$ be the solution of \eqref{P1} and \eqref{P2}, respectively. We define the reduced velocity space $V_r$ and pressure space $Q_r$  by
\begin{align}\label{PODbasis1}
\begin{split}
V_r = \textup{span}\{\widetilde \varphi_1, \widetilde\varphi_2,\ldots, \widetilde\varphi_{r_u}\},\qquad Q_r = \{\widetilde \psi_1, \widetilde\psi_2,\ldots, \widetilde\psi_{r_p}\}.
\end{split}
\end{align}
One interesting property is that $V_r$ is weakly divergence free due to  
\begin{align}\label{Vh_div}
V_r \subset \textup{span}\{\mathfrak{u}_h^1, \mathfrak{u}_h^2,\ldots, \mathfrak{u}_h^\ell\} \subset  V_{h}^{\textup{div}}:=\left\{v_{h} \in V_{h},\quad  b\left(v_{h}, q_{h}\right)=0, \quad  \forall q_{h} \in Q_{h}\right\}.	
\end{align}

\subsection{Velocity only ROM}\label{VonlyROM}
Using the space $V_r$ we  construct the BDF2-ROM scheme,   given $u_{r}^{0} =0$, for each $n=1,2,\ldots, N_T$, we find velocity $u_{r}^{n} \in V_{r}$ satisfying
\begin{align}\label{Velocity_onlyROM}
\left(\partial_t^+ u_r^n, v_r\right)+a\left( u_{r}^{n}, v_r\right)=\left(f, v_r\right) \quad \forall v_r \in V_r.
\end{align}
The terms involving the pressure have dropped out of \eqref{Velocity_onlyROM} due to \eqref{Vh_div}.

\subsection{Pressure recovery}  
The ROM	 \eqref{Velocity_onlyROM} only computes the velocity.  In this section, we use the velocity $u_r^n, u_r^{n-1}, u_r^{n-2}$ and the momentum equation \eqref{moment} to recover the pressure from the reduced pressure  space $Q_r$. As we discussed in \Cref{VonlyROM}, the pressure was dropped from the formulation \eqref{Velocity_onlyROM} due to $V_r \subset V_{h}^{\textup{div}}$. Hence, to recover the pressure, we need to use some test functions that do not belong to $ V_{h}^{\textup{div}}$. From Hilbert space theory, the function space $V_{h}$ can be decomposed into the orthogonal subspaces:
$$
V_{h}=V_{h}^{\textup{div}} \oplus\left(V_{h}^{\textup{div}}\right)^{\perp},
$$
where the orthogonality is in the sense of the $V$ inner product.

The momentum equation recovery (MER) approach for recovering the pressure by using the weak form of the momentum equation via a Petrov-Galerkin projection, i.e., given the velocity only ROM solution $u_{r}^{n}, u_{r}^{n-1}, u_{r}^{n-2}$, determined by \eqref{Velocity_onlyROM}, find $p_{r}^{n} \in Q_r$ satisfying
\begin{align}\label{pressure_r20}
b(s_h, p_r^n)=\left(f, s_h\right)- \left(\partial_t^+ u_r^n, s_h\right)  - a\left( u_{r}^{n} ,  s_h\right), \quad \forall s_h \in \mathcal S_h.
\end{align}		

To recover the pressure $p_r^n$ from \eqref{pressure_r20}, we 
would require that the matrix form of $b(s_h, p_r^n)$ is square and invertible. In other words, we need to determine  the test space $\mathcal S_h$ such that it is inf-sup stable with respect to the  reduced pressure space $Q_r$. Following \cite{MR4133479,MR3894167},  we consider the discrete inf-sup condition \eqref{Discrtr_inf_Sup} by replacing the pressure finite element space with the ROM space $Q_r$:
\begin{align}\label{new_inf_sup}
\inf _{p_r \in Q_r \backslash\{0\}} \sup _{s_{h} \in \mathcal S_{h} \backslash\{0\}} \frac{b\left(s_{h}, p_r\right)}{\left\| s_{h}\right\|_V\|p_r\|_Q}\ge \beta_h.	
\end{align}

This can be done by using the Riesz representation in $V_{h}$ of the linear functional $b(\cdot, p_r)$, i.e., find $s_{h} \in V_{h}$ such that
\begin{align}\label{Ritz_projection}
a\left( s_{h},  v_{h}\right)=b\left( v_{h}, p_r\right),  \quad \forall v_{h} \in V_{h}.	
\end{align}

Solving \eqref{Ritz_projection} for each basis function of $Q_r$ yields a set of basis functions $\left\{\zeta_{i}\right\}_{i=1}^{r_p}$. Letting
\begin{align}\label{def_Sh}
\mathcal S_h:=\operatorname{span}\left\{\zeta_{i}\right\}_{i=1}^{r_p} \subset\left(V_{h}^{\textup{div}}\right)^{\perp} \subset V_{h}.
\end{align}

We note that the test function $s_h\in \mathcal S_h\subset \left(V_{h}^{\textup{div}}\right)^{\perp}$ and the velocity solution $ u_r^n\in V_r\subset V_{h}^{\textup{div}}$, then $ a\left( u_{r}^{n},  s_h\right)=0$. In other words, we 
find $p_{r}^{n} \in Q_r$ satisfying
\begin{align}\label{pressure_r3}
b(s_h, p_r^n)=\left(f, s_h\right)- \left(\partial_t^+ u_r^n, s_h\right),  \quad \forall s_h \in \mathcal S_h.
\end{align}

\subsection{Implementation}\label{impleu}
First, we  compute $\left(\mathfrak{u}_h^i, \mathfrak{p}_h^i\right)$ by solving \eqref{Generate_basis}. Let $ \mathcal P^k(K)$ denote the set of polynomials of degree at most $k$ on an element $K$. We define
\begin{align*}
V_h &= \left\{v_h\in C(\bar\Omega)\big| v_h|_K\in \mathcal P^{k+1}(K)\right\}=\operatorname{span}\left\{\varphi_{1}, \ldots, \varphi_{N_u}\right\},\\
Q_h &= \left\{q_h\in C(\bar\Omega)\big| q_h|_K\in \mathcal P^{k}(K)\right\}=\operatorname{span}\left\{\psi_{1}, \ldots, \psi_{N_p}\right\}.
\end{align*} 
Define 
\begin{align}\label{MassStiffnessload}
M_{i j}=(\varphi_{j}, \varphi_{i}), \quad     A_{i j}= a(\varphi_{j}, \varphi_{i}), \quad   B_{ij} = b( \varphi_j, \psi_i), \quad   W_{ij} = ( \psi_j, \psi_i), \quad b_i = (f, \varphi_i).
\end{align}
Let $\texttt{u}_{i}$ and $\texttt{p}_i$ be the coefficients of $\mathfrak{u}_h^i$ and $\mathfrak{p}_h^i$, $1\le i\le \ell$,  i.e., 
\begin{align}\label{Coefficients}
\mathfrak{u}_h^i = \sum_{j=1}^{N_u} \left(\texttt{u}_{i}\right)_j \varphi_j, \qquad \textup{and} \qquad 	\mathfrak{p}_h^i = \sum_{j=1}^{N_p} \left(\texttt{p}_{i}\right)_j \psi_j,
\end{align}
where $(\alpha)_j$ denotes the $j$-th component of the vector $\alpha$. Then substitute  \eqref{Coefficients} into \eqref{Generate_basis} we obtain
\begin{align}\label{Obtain_up}
\left[\begin{array}{cccc}
A&B\\[0.2cm]
-B^\top&O\\
\end{array}\right]  \left[\begin{array}{cccc}
\texttt{u}_{1}\\[0.2cm]
\texttt{p}_{1}\\
\end{array}\right] = \left[\begin{array}{cccc}
b\\[0.2cm]
O\\
\end{array}\right],\qquad \left[\begin{array}{cccc}
A&B\\[0.2cm]
-B^\top&O\\
\end{array}\right]  \left[\begin{array}{cccc}
\texttt{u}_{i}\\[0.2cm]
\texttt{p}_{i}\\
\end{array}\right] = \left[\begin{array}{cccc}
M\texttt{u}_{i-1}\\[0.2cm]
O\\
\end{array}\right],\qquad 2\le i\le \ell,
\end{align}

The algebraic system \eqref{Obtain_up}   has defects. One of which is that the matrix is
singular so that solving \eqref{Obtain_up}    is usually impossible. There are two common ways to treat this issue, the first one is to  fix the pressure at one point and the second is to impose zero  mean by introducing a Lagrange multiplier. In this paper, we shall use the first approach.

After we solved \eqref{Generate_basis}, we collect the coefficients of $\mathfrak{u}_h^i$ and $\mathfrak{p}_h^i$.  Define the matrices $U_\ell$ and $P_\ell$ by
\begin{align}\label{Coefficients_Matrix}
U_\ell = [\texttt{u}_{1} \mid \texttt{u}_{2} \mid \ldots\mid \texttt{u}_{\ell}]\in \mathbb R^{N_u\times \ell},\qquad P_\ell = [\texttt{p}_{1} \mid \texttt{p}_{2} \mid \ldots\mid \texttt{p}_{\ell}]\in \mathbb R^{N_p\times \ell}.
\end{align}

Next, we shall solve the  optimization problems \eqref{P1} and \eqref{P2} to find the reduced velocity and pressure space.  The approach taken is the same as in \cite{WalingtonWeberZhang}, hence we only list the essential steps. 
\begin{itemize}
	\item[(1)] Let 	
	\begin{align}\label{def_Kell}
	K_\ell  = U_\ell ^\top A U_\ell, \qquad 		G_\ell  = P_\ell ^\top W P_\ell.
	\end{align}

	\item [(2)] Let $\lambda_k(K_\ell)$ and $x_k $ be the $k$-th eigenvalue  and eigenvector of $K_\ell$; $\lambda_k(G_\ell)$ and $y_k $ be the $k$-th eigenvalue  and eigenvector of $G_\ell$.

	\item[(3)] Give a tolerance $\texttt{tol}$, we find the minimal $r_u$ and $r_p$  such that
	\begin{align*}
	\dfrac{\sum_{k=1}^{r_u} \lambda_k(K_\ell)}{\sum_{k=1}^{\ell} \lambda_k(K_\ell)} \ge 1-\texttt{tol},\qquad \dfrac{\sum_{k=1}^{r_p} \lambda_k(G_\ell)}{\sum_{k=1}^{\ell} \lambda_k(G_\ell)} \ge 1-\texttt{tol}.
	\end{align*}
	
	\item[(4)]  Let $\widetilde x_k  = \frac{1}{\sqrt{\lambda_k(K_\ell)}} U_\ell x_k$, $k=1,2\ldots, r_u$; $\widetilde y_k  = \frac{1}{\sqrt{\lambda_k(G_\ell)}} P_\ell y_k$, $k=1,2\ldots, r_p$. We note  that $\widetilde x_k$ and  $\widetilde y_k$ are the coefficients of $\widetilde \varphi_k$ and $\widetilde \psi_k$, respectively.

	\item[(5)] Define $Q_u = [\widetilde x_1\mid \widetilde x_2\mid \cdots\mid \widetilde x_{r_u}]$ and $Q_p = [\widetilde y_1\mid \widetilde y_2\mid \cdots\mid \widetilde y_{r_p}]$. Therefore, 
	\begin{align}\label{QuQp}
	\begin{split}
	Q_u &= U_\ell [x_1\mid x_2\mid \ldots| x_{r_u}]\begin{bmatrix}\frac{1}{\sqrt{\lambda_1(K_\ell)}} & & \\ & \ddots & \\ & & \frac{1}{\sqrt{\lambda_{r_u}(K_\ell)}}\end{bmatrix} \in \mathbb R^{N_u\times r_u},\\
	Q_p &= P_\ell [y_1\mid y_2\mid \ldots| y_{r_p}]\begin{bmatrix}\frac{1}{\sqrt{\lambda_1(G_\ell)}} & & \\ & \ddots & \\ & & \frac{1}{\sqrt{\lambda_{r_p}(G_\ell)}}\end{bmatrix} \in \mathbb R^{N_p\times r_p}.	
	\end{split}
	\end{align}
\end{itemize}

Once we obtained the reduced pressure space, we then compute the basis of $\mathcal S_h$. In other words,  we solve \eqref{Ritz_projection} with $p_r = \widetilde \psi_j$. Assume  that  $s_{j,h}\in V_h$ is the solution, $1\le j\le r_p$, and  let $s_j$ be the coefficient of $s_{j,h}$ under the finite element basis $\{\varphi_k\}_{k=1}^{N_u}$, i.e., $s_{j,h} = \sum_{k=1}^{N_u} (s_{j})_k\varphi_k$, $1\le j\le r_p$. Then the matrix equation of \eqref{Ritz_projection} is
\begin{align*}
A\underbrace{\left[s_1~\big|~s_2~\big| \cdots \big| s_{r_p}\right]}_{S}= BQ_p.
\end{align*}

Then the reduced velocity space $V_r$, the reduced pressure spaces $Q_r$, and the space $\mathcal S_h$ are defined by
\begin{align*}
V_r &= \textup{span}\{\widetilde \varphi_1, \widetilde \varphi_2,\ldots,\widetilde \varphi_{r_u}\}, \qquad\; \textup{with} \;\; \widetilde \varphi_{k} = \sum_{j=1}^{r_u} Q_u(j,k)\varphi_j,\\
Q_r &= \textup{span}\{\widetilde \psi_1, \widetilde \psi_2,\ldots,\widetilde \psi_{r_u}\}, \qquad\; \textup{with} \;\; \widetilde \varphi_{k} = \sum_{j=1}^{r_p} Q_p(j,k)\psi_j,\\
\mathcal S_h &= \textup{span}\{s_{1,h},  s_{2,h},\ldots,s_{r_p,h}\}, \quad \textup{with} \;\;  s_{k,h}= \sum_{j=1}^{r_u} S(j,k)\varphi_j.
\end{align*}

The next step is to give the matrix form of \eqref{Velocity_onlyROM} and \eqref{pressure_r3}. Since $u_r^n\in V_r$ and $p_r^n\in Q_{r}$ hold, we then make the Galerkin ansatz of the form
\begin{align}\label{Reduced_Galerkin3}
u_r^n=\sum_{j=1}^{r_u} \left(\alpha_r^n\right)_j { \widetilde \varphi}_{j},\qquad p_r^n=\sum_{j=1}^{r_p} \left(\beta_r^n\right)_j {\widetilde \psi}_{j}.
\end{align}
We insert \eqref{Reduced_Galerkin3} into \eqref{Velocity_onlyROM} and \eqref{pressure_r3}  to obtain the following linear matrix form:
\begin{gather}\label{semidiscrete25}
\begin{split}
M_{r} \partial_t^+  \alpha_r^{n}  + A_{r} \alpha_r^{n}=b_{r},\quad \alpha_r^0 = 0,\\ 
B_r  \beta_r^{n}  = \widetilde b_{r} -  W_r\partial_t^+ \alpha_r^n ,
\end{split}
\end{gather}
where
\begin{gather*}
M_r = Q_u^\top M Q_u \in \mathbb R^{r_u\times r_u}, \qquad A_r = Q_u^\top A Q_u \in \mathbb R^{r_u\times r_u}, \qquad b_r = Q_u^\top b \in \mathbb R^{r_u},\\
B_r = S^\top  B Q_p\in \mathbb R^{r_p\times r_p},\qquad  W_r = (BQ_p)^\top MQ_u\in \mathbb R^{r_p\times r_u}, \qquad \widetilde b_r = S^\top b\in \mathbb R^{r_p}.
\end{gather*}

The final step is to return the solution of the ROM  to the FOM. In other words, we shall express the solutions $u_r^n$ and $p_r^n$ under the finite element basis functions. By \eqref{Reduced_Galerkin3} and \eqref{QuQp} we have 
\begin{align*}
u_r^n=\sum_{j=1}^{r_u} \left(\alpha_r^n\right)_j { \widetilde \varphi}_{j}  = \sum_{j=1}^{r_u} \left(\alpha_r^n\right)_j \sum_{i=1}^{N_u} (\widetilde x_j)_i \varphi_i = \sum_{i=1}^{N_u} (Q_u \alpha_r^n)_i\varphi_i,\\
p_r^n=\sum_{j=1}^{r_p} \left(\beta_r^n\right)_j {\widetilde \psi}_{j}  = \sum_{j=1}^{r_p} \left(\beta_r^n\right)_j \sum_{i=1}^{N_p} (\widetilde y_j)_i \psi_i = \sum_{i=1}^{N_p} (Q_p \beta_r^n)_i\psi_i.
\end{align*}

That is to say,  the solution $u_r^n$, in terms of the  finite element basis $\left\{\varphi_{1}, \ldots, \varphi_{N_u}\right\}$, the coefficient is $Q_u\alpha_r^n$;  the solution $p_r^n$, in terms of the  finite element basis $\left\{\psi_{1}, \ldots, \varphi_{N_p}\right\}$, the coefficient is $Q_p\beta_r^n$.

Now, we summarize the above discussions in \Cref{algorithm0}.

\begin{algorithm}[H]
	\caption{}
	\label{algorithm0}
	{\bf{Input}:}  \texttt{tol}, $\ell$, $M$, $W$,  $A$, $B$, $b$
	\begin{algorithmic}[1]
		\State Solve $
		\left[\begin{array}{cccc}
		A&B\\[0.2cm]
		-B^\top&O\\
		\end{array}\right]  \left[\begin{array}{cccc}
		\texttt{u}_1\\[0.2cm]
		\texttt{p}_1\\
		\end{array}\right] = \left[\begin{array}{cccc}
		b\\[0.2cm]
		O\\
		\end{array}\right]$;
		\For{$i=2$ to $\ell$}
		\State Solve $
		\left[\begin{array}{cccc}
		A&B\\[0.2cm]
		-B^\top&O\\
		\end{array}\right]  \left[\begin{array}{cccc}
		\texttt{u}_i\\[0.2cm]
		\texttt{p}_i\\
		\end{array}\right] = \left[\begin{array}{cccc}
		M\texttt{u}_{i-1}\\[0.2cm]
		O\\
		\end{array}\right]$;		
		\EndFor
		\State Set $U = [\texttt{u}_{1} \mid \texttt{u}_{2} \mid \ldots\mid \texttt{u}_{\ell}]$;  $P = [\texttt{p}_{1} \mid \texttt{p}_{2} \mid \ldots\mid \texttt{p}_{\ell}]$;
		\State Set $K = U^\top A U$; \quad   $[X, \Lambda_1] = \textup{eig}(K)$; \quad $G = P^\top W P$;\quad   $[Y, \Lambda_2] = \textup{eig}(G)$;
		\State Find minimal $r_u$ and $r_p$ such that  $\frac{\sum_{i=1}^{r_u}\Lambda_1(i,i)}{\sum_{i=1}^\ell\Lambda_1(i,i)}\ge 1-\texttt{tol}$ and $\frac{\sum_{i=1}^{r_p}\Lambda_2(i,i)}{\sum_{i=1}^\ell\Lambda_2(i,i)}\ge 1-\texttt{tol}$ . 
		\State Set $Q_u = UX(:,1:r_u)(\Lambda_1(1:r_u,1:r_u))^{-1/2}$ and $Q_p  = PY(:,1:r_p)(\Lambda_2(1:r_p,1:r_p))^{-1/2}$;
		\State Solve $	AS = BQ_p$;		
		\State  		 Set $M_r = Q_u^\top M Q_u$; $A_r = Q_u^\top A Q_u$; $b_r = Q_u^\top b$;	$B_r = S^\top  B Q_p$;  $W_r = (BQ_p)^\top MQ_u$; $\widetilde b_r = S^\top b$; 
		\For {$n = 1$ to $N_T$}	
		\State Solve $ 		M_{r} \partial_t^+  \alpha_r^{n}  + A_{r} \alpha_r^{n}=b_{r}$;
		\State Solve $		B_r  \beta_r^{n}  = \widetilde b_{r} -  W_r\partial_t^+ \alpha_r^n $;
		\EndFor
		\State {\bf return} $Q_u$,  $Q_p$, $\{\alpha_r^n\}_{n=1}^{N_T}$, $\{\beta_r^n\}_{n=1}^{N_T}$
	\end{algorithmic}
\end{algorithm}

\section{Theoretical analysis}
It seems that the dimension of the reduced velocity space $V_r$ and pressure space $Q_r$ depends on the eigenvalues of $K_\ell$ and $G_\ell$, respectively. However, in this section we prove that  only the eigenvalues of $K_\ell$  determine the main computational cost and the dimension of $V_r$ and $Q_r$. Furthermore, we show that the eigenvalues of  $K_\ell$ are exponentially decay.

Our discussion relies on the  discrete eigenvalue problem of the Stokes equation. Let $(\phi_h, \chi_h, \lambda_h)$, with $\phi_h \neq 0$ and $\lambda_h \in \mathbb{R}$ be the solution of 
\begin{subequations}\label{Weak_Stokes_Eigenvalue}
	\begin{align}
	a(\phi_h, v_h) + b(v_h, \chi_h) &=\lambda_h (\phi_h, v_h) \qquad\qquad\qquad\quad   \forall v_h\in V_h, \label{Weak_Stokes_Eigenvalue1}\\
	b(\phi_h, q_h)&=0   \qquad\qquad\qquad\qquad \qquad  \;\;\;\forall q_h\in Q_h.\label{Weak_Stokes_Eigenvalue2}
	\end{align}	
\end{subequations}

It is well known that the  discrete Stokes eigenvalue problem \eqref{Weak_Stokes_Eigenvalue} has a finite sequence of eigenvalues and  eigenfunctions
\begin{gather*}
0<\lambda_{1,h} \leq \lambda_{2,h} \leq \ldots\le \lambda_{N_u,h}, \quad  
\left(\phi_{1,h}, \chi_{1,h}\right),	\left(\phi_{2,h}, \chi_{2,h}\right), \ldots,	\left(\phi_{N_u,h}, \chi_{N_u,h}\right),\quad (\phi_{i,h},\phi_{j,h})_V = \delta_{ij}.
\end{gather*}

Define $\mathcal A_h: V_h\times Q_h \to V_h\times Q_h$ by
\begin{align}\label{bilinear_Ah}
\mathscr A ((u_h, p_h), (v_h,q_h)) =  \left(\mathcal A_h (u_h, p_h), (v_h,q_h)\right)\quad \text { for all } (v_{h}, q_h) \in V_{h}\times Q_h.	
\end{align}
It is easy  to verify that  $\mathcal A_h^{-1}: V_h\times Q_h \to V_h\times Q_h$ exists, and
\begin{align}\label{EigAhinverse}
\mathcal A_h^{-1}\left[\begin{array}{cc}
\phi_{i,h} \\[0.2cm]
0
\end{array}\right] = \lambda_{i,h}^{-1}\left[\begin{array}{cc}
\phi_{i,h} \\[0.2cm]
\chi_{i,h}
\end{array}\right].
\end{align}
By the definition of $\left(	\mathfrak{u}_h^1, 
\mathfrak{p}_h^1\right)$ in \eqref{Generate_basis}, for all $ (v_{h}, q_h) \in V_{h}\times Q_h$ we have 
\begin{align}\label{bilinear_Ah1}
\left(\mathcal A_h (\mathfrak{u}_h^1 , \mathfrak{p}_h^1 ), (v_h,q_h)\right) = ((f,0), (v_h,q_h)) = ((\Pi f,0), (v_h,q_h)) ,	
\end{align}
where $\Pi: [L^2(\Omega)]^d \to V_h$ be the standard $L^2$ projection. Therefore
\begin{align}\label{EigAhinverse2}
\left[\begin{array}{cc}
\mathfrak{u}_h^1 \\[0.2cm]
\mathfrak{p}_h^1
\end{array}\right] = \mathcal A_h^{-1}	 \left[\begin{array}{cc}
\Pi f \\[0.2cm]
0
\end{array}\right]. 
\end{align}
For $i=2,3\ldots, \ell$, we have
\begin{align}\label{EigAhinverse3}
\left[\begin{array}{cc}
\mathfrak{u}_h^i \\[0.2cm]
\mathfrak{p}_h^i
\end{array}\right] = \mathcal A_h^{-1}	 \left[\begin{array}{cc}
\mathfrak{u}_h^{i-1}\\[0.2cm]
0
\end{array}\right]. 
\end{align}

Obviously,  $\{\left(	\mathfrak{u}_h^1, 
\mathfrak{p}_h^1\right), \left(	\mathfrak{u}_h^2, 
\mathfrak{p}_h^2\right), \ldots, \left(	\mathfrak{u}_h^\ell, 
\mathfrak{p}_h^\ell\right)\}$ is not a Krylov sequence. However, we can show that $\{\mathfrak{u}_h^1,\mathfrak{u}_h^2, \ldots, \mathfrak{u}_h^\ell\}$ is a Krylov sequence.
\begin{lemma}\label{KKrylovs}
	The sequence  $\{\mathfrak{u}_h^1,\mathfrak{u}_h^2, \ldots, \mathfrak{u}_h^\ell\}$ is a Krylov sequence.
\end{lemma}
\begin{proof}
	Let $f_h$ be the standard $L^2$ projection of $f$ in $V_h$ and assume that 
	\begin{align*}
	f_h = \sum_{j=1}^{N_u} c_j \phi_{j,h}.
	\end{align*}
	By \eqref{EigAhinverse} and \eqref{EigAhinverse2} we have 
	\begin{align*}
	\left[\begin{array}{cc}
	\mathfrak{u}_h^1 \\[0.2cm]
	\mathfrak{p}_h^1
	\end{array}\right] = \mathcal A_h^{-1}	 \left[\begin{array}{cc}
	\Pi f \\[0.2cm]
	0
	\end{array}\right] =  \mathcal A_h^{-1}	 \left[\begin{array}{cc}
	\displaystyle\sum_{i=1}^{N_u} c_i \phi_{i,h}\\[0.4cm]
	0
	\end{array}\right] = \sum_{i=1}^{N_u} c_i\lambda_{i,h}^{-1} \left[\begin{array}{cc}
	\phi_{i,h}\\[0.4cm]
	\chi_{i,h}
	\end{array}\right].
	\end{align*}
	By the same argument as above, we formally have 
	\begin{align}\label{krylovw}
	\left[\begin{array}{cc}
	\left(	\mathfrak{u}_h^1, 
	\mathfrak{p}_h^1\right)  \\[0.3cm]
	\left(	\mathfrak{u}_h^2, 
	\mathfrak{p}_h^2\right)\\
	\vdots\\
	\left(	\mathfrak{u}_h^\ell, 
	\mathfrak{p}_h^\ell\right)
	\end{array}\right] = \left[\begin{array}{cccc}
	c_1\mu_1&c_2\mu_2&\cdots&c_{N_u}\mu_{N_u}\\[0.2cm]
	c_1\mu_1^2&c_2\mu_2^2&\cdots&c_{N_u}\mu_{N_u}^2\\[0.2cm]
	\vdots&\vdots&\cdots&\vdots\\
	c_1\mu_1^\ell&c_2\mu_2^\ell&\cdots&c_{N_u}\mu_{N_u}^\ell\\[0.2cm]
	\end{array}\right] \left[\begin{array}{cc}
	(\phi_{1,h},\chi_{1,h}) \\[0.2cm]
	(\phi_{2,h},\chi_{2,h})\\
	\vdots\\
	(\phi_{N_u,h},\chi_{N_u,h})
	\end{array}\right].  
	\end{align}
	
	We define $\mathscr B_h: V_h\to V_h$ by
	\begin{align}\label{def_B}
	\mathscr B_h v_h = \sum_{j=1}^{N_u} \mu_j (v_h, \phi_{j,h})_V \phi_{j,h}.
	\end{align}
	It is easy to show that $\mathscr B_h: V_h\to V_h$ is bounded with the $V$-norm, then
	\begin{align*}
	\{\mathfrak{u}_h^1, \mathfrak{u}_h^2,\ldots, \mathfrak{u}_h^\ell\} = 	\{\mathscr B_h f_h, \mathscr B_h^2 f_h,\ldots, \mathscr B_h^\ell f_h\}.
	\end{align*}
	
\end{proof}

For each $\mathfrak{u}_h^{i-1}\in V_h$,  we can determine $\mathfrak{p}_h^{i}\in Q_h$ by \eqref{EigAhinverse3}, this determines a linear operator $\mathscr C_h: V_h \to Q_h$ by 
\begin{align*}
\mathfrak{p}_h^{i}	= \mathscr C_h \mathfrak{u}_h^{i-1}, \quad i=1,2,\ldots, \ell.
\end{align*} 
This implies
\begin{align}\label{Krylov_ph}
\{\mathfrak{p}_h^1, \mathfrak{p}_h^2,\ldots, \mathfrak{p}_h^\ell\} = 	\{\mathscr C_h\mathscr B_h f_h, \mathscr C_h\mathscr B_h^2 f_h,\ldots, \mathscr C_h\mathscr B_h^\ell f_h\}.
\end{align}
\begin{remark}
	The sequence  $\{\mathfrak{p}_h^1,\mathfrak{p}_h^2, \ldots, \mathfrak{p}_h^\ell\}$ is \emph{not} a Krylov sequence. 
\end{remark}

Let $r\le \ell$ be the largest number such that $
\{\mathfrak{u}_h^1, \mathfrak{u}_h^2,\ldots, \mathfrak{u}_h^r\}$
is linear independent. Although $\{\mathfrak{p}_h^1,\mathfrak{p}_h^2, \ldots, \mathfrak{p}_h^\ell\}$ is \emph{not} a Krylov sequence, by \eqref{Krylov_ph} we know that
\begin{align*}
\textup{span}	\{\mathfrak{p}_h^1, \mathfrak{p}_h^2,\ldots, \mathfrak{p}_h^r\} =  \textup{span}	\{\mathfrak{p}_h^1, \mathfrak{p}_h^2,\ldots, \mathfrak{p}_h^r,\ldots\}.
\end{align*}
This implies that the dimension of $Q_r$ is no larger than $r$. In other words, the eigenvalues of $K_\ell$  determine not only the dimension of $V_r$, but also the up-bound dimension of  $Q_r$.

Furthermore,  the matrices $K_j$, $j=1,2\ldots, r$ are positive definite and $K_{r+1}$ is  positive semi-definite. Hence, we only to compute the minimal eigenvalue of the matrices $K_1, K_2,\ldots$. Once the minimal eigenvalue of some matrix is zero, we then stop. 

In practice, we terminate the process if the minimal eigenvalue of some matrix is small. 

Following the same arguments in \cite{WalingtonWeberZhang}, we can prove that the matrices $K_\ell$ in \eqref{def_Kell} is  Hankel type matrix, i.e.,  each ascending skew-diagonal from left to right is constant. 
Therefore, to assemble the matrix $K_r$, we only need the matrix $K_{r-1}$ and to compute $\texttt{u}_{r-1}^\top A \texttt{u}_{r}$ and  $ \texttt{u}_{r}^\top A \texttt{u}_{r}$. 
Next we summarize the above discussion in \Cref{algorithm1}.
\begin{algorithm}[H]
	\caption{ (Get $Q_u$ and $Q_p$)}
	\label{algorithm1}
	{\bf{Input}:}  tol, $\ell$, $M$, $W$,  $A$, $B$, $b$
	\begin{algorithmic}[1]
		\State Solve $
		\left[\begin{array}{cccc}
		A&B\\[0.2cm]
		-B^\top&O\\
		\end{array}\right]  \left[\begin{array}{cccc}
		\texttt{u}_1\\[0.2cm]
		\texttt{p}_1\\
		\end{array}\right] = \left[\begin{array}{cccc}
		b\\[0.2cm]
		O\\
		\end{array}\right]$;
		\State Let $K_1 = \texttt{u}_{1}^\top A\texttt{u}_{1}$;
		\For{$i=2$ to $\ell$}
		\State Solve $
		\left[\begin{array}{cccc}
		A&B\\[0.2cm]
		-B^\top&O\\
		\end{array}\right]  \left[\begin{array}{cccc}
		\texttt{u}_i\\[0.2cm]
		\texttt{p}_i\\
		\end{array}\right] = \left[\begin{array}{cccc}
		M\texttt{u}_{i-1}\\[0.2cm]
		O\\
		\end{array}\right]$;		
		
		\State Get $\alpha = [K_{i-1}(i-1,2:i-1) \mid \texttt{u}_{i-1}^\top A \texttt{u}_{i}]$ and  $\beta = \texttt{u}_{i}^\top A \texttt{u}_{i}$;
		\State $K_i = \left[\begin{array}{cc}
		K_{i-1}&\alpha^\top\\
		\alpha&\beta\\
		\end{array}\right]$;
		\State $[X, \Lambda_1] = \textup{eig}(K_i)$;
		\If{$\Lambda_1(i,i)\le \textup{tol}$}
		\State break;
		\EndIf	
		\EndFor
		\State Set $U = [\texttt{u}_{1} \mid \texttt{u}_{2} \mid \ldots\mid \texttt{u}_{i}]$;  $P = [\texttt{p}_{1} \mid \texttt{p}_{2} \mid \ldots\mid \texttt{p}_{i-1}]$;
		\State Set $G = P^\top W P$;\quad   $[Y, \Lambda_2] = \textup{eig}(G)$;
		\State Find minimal  $r_p$ such that  $\frac{\sum_{j=1}^{r_p}\Lambda_2(j,j)}{\sum_{j=1}^{i-1}\Lambda_2(j,j)}\ge 1-\textup{tol}$;
		\State Set $Q_p  = PY(:,1:r_p)(\Lambda_2(1:r_p,1:r_p))^{-1/2}$;
		\State Set $Q_u = UX(:,1:i-1)(\Lambda_1(1:i-1,1:i-1))^{-1/2}$;
		\State {\bf return} $Q_u,  Q_p$
	\end{algorithmic}
\end{algorithm}

Next, we  give the estimation of the eigenvlaues of $K_r$. The proof of the following \Cref{UAU} is the same with the proof of \cite{WalingtonWeberZhang}, hence we omit the details here.
\begin{theorem}\label{UAU}
	Let $\lambda_1 ({K_{r}})\ge \lambda_2({K_{r}})\ge \ldots\ge \lambda_r({K_{r}})> 0$ be the eigenvalues of $K_{r}$, then
	\begin{align}\label{all_eigs}
	\lambda_{2k+1} ({K_{r}}) \le 16\left[\exp \left(\frac{\pi^{2}}{4 \log (8\lfloor r / 2\rfloor / \pi)}\right)\right]^{-2k+2}\lambda_1 ({K_{r}}), \qquad 2k+1\le r.
	\end{align}
	Moreover, the minimal eigenvalue of $K_r$ satisfies
	\begin{align}\label{last_eigs}
	\lambda_{\min} ({K_{r}})  \le C  (2r-1) \|f\|_{V'}^2 \exp\left(-\dfrac{7(r+1)}{2}\right).
	\end{align}
\end{theorem}

\begin{example}\label{Example2}
	We use the same problem data as in the \Cref{Example0} and take $h=1/100$. We report  all the eigenvalues of $K_{10}$ and $G_{10}$ in \Cref{2DEigs}. It is clear that the eigenvalues of both $K_{10}$ and $G_{10}$ are exponentially decay. This matches our theoretical result in \Cref{UAU}.
	\begin{figure}[tbh]
		\centerline{
			\hbox{\includegraphics[width=3in]{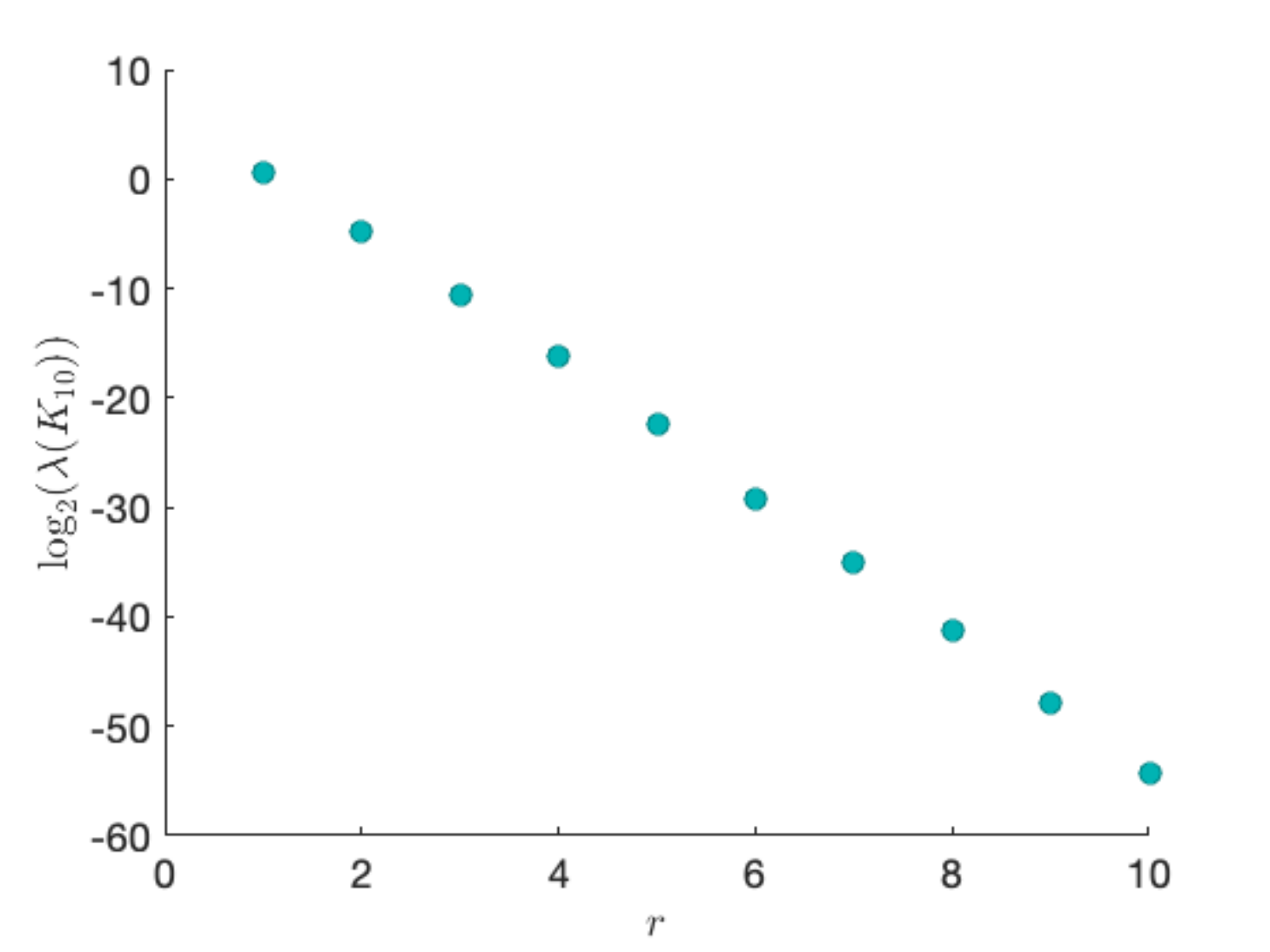}}
			\hbox{\includegraphics[width=3in]{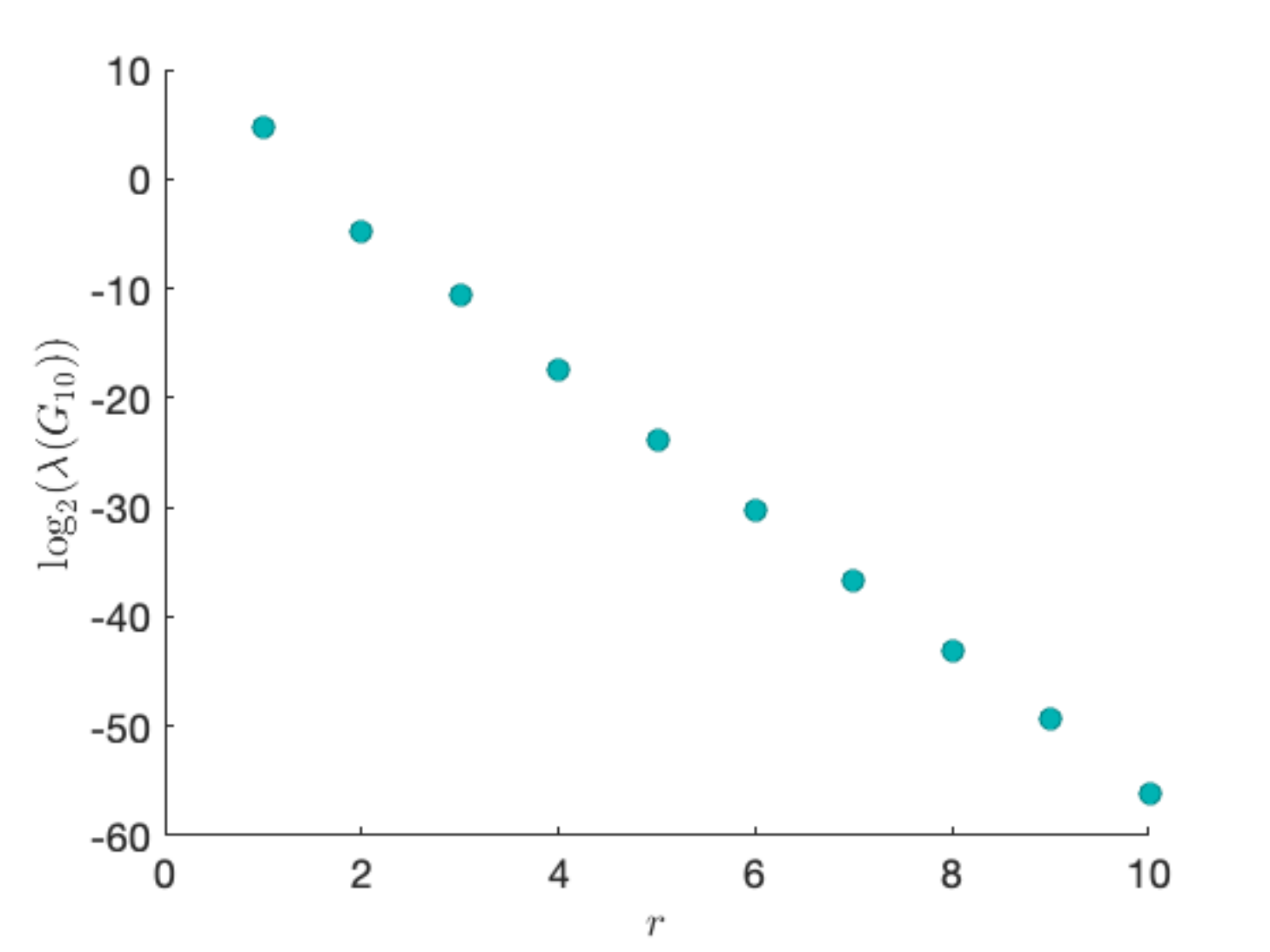}}}
		\caption{The eigenvalues of the matrices $K_{10}$ and $G_{10}$.}
		\label{2DEigs}
		\centering
	\end{figure}
	
\end{example}

Finally, we give the full implementation of \Cref{Velocity_onlyROM}.
\begin{algorithm}[H]
	\caption{}
	\label{algorithm2}
	{\bf{Input}:}  tol, $\ell$, $N_T$, $\Delta t$, $M$, $W$,  $A$, $B$, $b$
	\begin{algorithmic}[1]
		\State $[Q_u, Q_p] = \texttt{GetQuQp}(\textup{tol}, \ell, M, W,  A, B, b)$;   \hspace{6cm} \% \Cref{algorithm1}	
		\State Solve $	AS = BQ_p$;		
		\State  		 Set $M_r = Q_u^\top M Q_u$; $A_r = Q_u^\top A Q_u$; $b_r = Q_u^\top b$;	$B_r = S^\top A  S$;  $W_r = (BQ_p)^\top MQ_u$; $\widetilde b_r = S^\top b$; 
		\For {$n = 1$ to $N_T$}	
		\State Solve $ 		M_{r} \partial_t^+  \alpha_r^{n}  + A_{r} \alpha_r^{n}=b_{r}$;
		\State Solve $		B_r  \beta_r^{n}  = \widetilde b_{r} -  W_r\partial_t^+ \alpha_r^n $;
		\EndFor
		\State {\bf return} $Q_u$,  $Q_p$, $\{\alpha_r^n\}_{n=1}^{N_T}$, $\{\beta_r^n\}_{n=1}^{N_T}$
	\end{algorithmic}
\end{algorithm}

\begin{example}\label{Example3}
	We revisit the \Cref{Example0} under the same problem data, mesh and time step.   We choose $\ell=5$, tol $= 10^{-14}$ in \Cref{algorithm2}.  We report the the dimension and the wall time of the ROM  in \Cref{table_2}. Comparing with \Cref{table_0}, we see that our ROM is much faster than  standard solvers. We also compute the $L^2$-norm error between the solutions of the FEM and the ROM at the final time $T=1$, the error is close to the machine error. This motivated us that the solutions of the FEM and of the ROM are the same if we take tol small enough in \Cref{algorithm2}. In \Cref{Error_analysis} we give a rigorous error analysis under an assumption on the source term $f$.
	\begin{table}[H]
		\centering
		{
			\begin{tabular}{c|c|c|c|c|c|c|c}
				\cline{1-8}
				
				$h$	& $1/2^1$ &$1/2^2$ 
				&$1/2^3$ 
				&$1/2^4$  &$1/2^5$ &$1/2^6$  &$1/2^7$  
				\\
				\cline{1-8}
				{$r$}
				& 	  5&   5&   5&   5&   5&   5&   5\\ 
				\cline{1-8}
				{Wall time }
				& 	  0.15&   0.03&   0.05 &   0.10&   0.49&   2.17&13.2\\ 
				\cline{1-8}
				{$\mathcal E_u$}
				& 3.76E-12&	  8.49E-11&   1.40E-13&   2.22E-13&   2.33E-13&  2.31E-13& 2.68E-13    \\ 
				\Xhline{0.1pt}
				
				{$\mathcal E_p$}
				& 9.48E-12&	  2.54E-09&   1.89E-13&   2.87E-13&   7.16E-13&  4.63E-12&   9.24E-13     \\ 
				\Xhline{0.1pt}
			\end{tabular}
		}
		\caption{\Cref{Example3}: The dimension and wall time (seconds) of the  ROM. The $L^2$-norm error between the solutions of the FEM and the ROM at the final time $T=1$.}\label{table_2}
	\end{table}		
\end{example}

\subsection{Error estimate of the velocity}\label{Error_analysis}
Next, we provide a fully-discrete convergence analysis of the new ROM for  the incompressible Stokes equation. 
Throughout this section, the  constant $C$  depends on the polynomial degree $k$, the domain, the shape regularity of the mesh and the problem data. But, it does not depend on the mesh size $h$, the time step $\Delta t$ and the dimension of the ROM.

First, we recall that $\Pi$ is the standard $L^2$ projection and $\{\phi_{j,h}\}_{j=1}^{N_u}$ are the eigenfuctions of \eqref{Weak_Stokes_Eigenvalue} corresponding to the eigenvalues $\{\lambda_{j,h}\}_{j=1}^{N_u}$.

Next, we give our main assumptions in this section:
\begin{assumption}\label{Assumption1}
	There exist  $\{c_j\}_{j=1}^\ell$ such that
	\begin{align}\label{Ass1}
	\Pi f = \sum_{j=1}^\ell c_j \phi_{m_jh}.
	\end{align}
\end{assumption}
\begin{assumption}\label{Assumption2}
	Regularity of the solution of \eqref{Weak_Stokes}: 
	\begin{align}\label{Ass2}
	u  \in H^2 \left(0,T; V\cap [H^{k+2}(\Omega)]^d\right), \; p  \in H^2 \left(0,T; Q\cap H^{k+1}(\Omega)\right).
	\end{align}
\end{assumption}
Now,  we state our main result in this section:
\begin{tcolorbox}
	\begin{theorem}\label{Main_res}
		Let $(u,p)$ be the solution of \eqref{Weak_Stokes} and $u_r^n$  be the solution of \eqref{Velocity_onlyROM} by setting $\textup{tol} = 0$ in \Cref{algorithm2}.  If Assumption \ref{Assumption1} and Assumption \ref{Assumption2} hold, then we have 
		\begin{align*}
		\left\|u(t_n) - u_r^n\right\|  \le C\left( h^{k+2} + (\Delta t)^2 \right).
		\end{align*}
	\end{theorem}			
\end{tcolorbox}

\subsection{Sketch the proof of \Cref{Main_res}}
To prove \Cref{Main_res}, we first bound the error
between the velocity of the PDE  \eqref{Strong_form_Stokes} and FEM
\eqref{CN_Stoeks}. Next we prove that  the velocity of 
\eqref{CN_Stoeks} and the ROM \eqref{Velocity_onlyROM} are exactly the same. Then we obtain  a bound on the error between the velocity of PDE \eqref{Strong_form_Stokes} and the ROM \eqref{Velocity_onlyROM}.

We begin by bounding the error between the velocity of  \eqref{CN_Stoeks} and PDE \eqref{Strong_form_Stokes}. 
\begin{lemma}\label{lemma1}
	Let $(u,p)$ and $ u_h^n$ be the solution of \eqref{Strong_form_Stokes} and \eqref{CN_Stoeks}, respectively. If Assumption \ref{Assumption2} holds, then we have
	\begin{align*}
	\left\|u(t_n) - u_h^n\right\|  \le C\left( h^{k+2} + (\Delta t)^2 \right).
	\end{align*}
\end{lemma}

The proof of \Cref{lemma1} is standard and we omit the proof. Next, we prove that  the velocity of 
\eqref{CN_Stoeks} and the ROM \eqref{Velocity_onlyROM} are exactly the same. 
\begin{lemma}\label{lemma2}
	Let $ u_h^n$ be the solution of \eqref{CN_Stoeks} and $u_r^n$ be the solution of \eqref{Velocity_onlyROM}  by setting $\textup{tol} = 0$ in \Cref{algorithm2}.  If Assumption \ref{Assumption1} holds, then for all $n=1,2,\ldots, N_T$ we have 
	\begin{align*}
	u_h^n  =  u_r^n.
	\end{align*}
\end{lemma}

As a consequence, \Cref{lemma1,lemma2}  give the proof of \Cref{Main_res}. 

\subsection{Proof of \Cref{lemma2}} 
Since the eigenvalue problem \eqref{Weak_Stokes_Eigenvalue} might have repeated eigenvalues. Without loss of generality, we assume that only $\phi_{m_1,h}$ and $\phi_{m_2,h}$ share the same eigenvalues $\lambda_{m_1,h} = \lambda_{m_2,h}$. Recall that $\mu_i = 1/\lambda_{i,h}$, then we have
\begin{align}\label{eig_ass}
\mu_{m_1} = 	\mu_{m_2}>\mu_{m_3} >\ldots > \mu_{m_\ell}. 
\end{align}

By \eqref{def_B} we know that $\phi_{i,h}$ is the eigenfuntion of  $\mathcal A_h^{-1}$ corresponding to the eigenvalue $\mu_i$. Similar to \eqref{krylovw} we formally have 
\begin{align}\label{Vandemonde}
\left[\begin{array}{cc}
(\mathfrak{u}_h^1,\mathfrak{p}_h^1)  \\[0.2cm]
(\mathfrak{u}_h^2,\mathfrak{p}_h^2)\\
\vdots\\
(\mathfrak{u}_h^\ell,\mathfrak{p}_h^\ell)
\end{array}\right] = \left[\begin{array}{cccc}
c_1 \mu_{m_1}&c_2 \mu_{m_2}&\cdots&c_\ell\mu_{m_\ell}\\[0.2cm]
c_1 \mu_{m_1}^2&c_2 \mu_{m_2}^2&\cdots&c_\ell\mu_{m_\ell}^2\\
\vdots&\vdots&\cdots&\vdots\\
c_1 \mu_{m_1}^{\ell}&c_2 \mu_{m_2}^\ell&\cdots&c_\ell\mu_{m_\ell}^\ell
\end{array}\right] \left[\begin{array}{cc}
(\phi_{m_1,h},\chi_{m_1,h})  \\[0.2cm]
(\phi_{m_2,h},\chi_{m_2,h}) \\
\vdots\\
(\phi_{m_\ell,h},\chi_{m_\ell,h}) 
\end{array}\right].
\end{align}

By the assumption \eqref{eig_ass}, the rank of the coefficient matrix in \eqref{Vandemonde} is $\ell -1$. 
By \Cref{KKrylovs} and the fact that $\{\phi_{m_i,h}\}_{i=1}^\ell$ are independent, we know $\ell-1$ is the maximum number such that $\{	\mathfrak{u}_h^1, 	\mathfrak{u}_h^2,\ldots, 	\mathfrak{u}_h^{\ell-1}\}$ are linear independent. 
This implies that the matrix $K_{\ell-1}$ (see \eqref{def_Kell}) is positive definite and $K_{\ell}$ is positive semi-definite. Therefore, if we set $\textup{tol} = 0$ in the \Cref{algorithm2}, then the reduced velocity space $V_r$ is given by
\begin{align*}
V_r =  \textup{span}\{\widetilde \varphi_1, \widetilde\varphi_2,\ldots, \widetilde\varphi_{\ell-1}\} =	\textup{span}\{	\mathfrak{u}_h^1, 	\mathfrak{u}_h^2,\ldots, 	\mathfrak{u}_h^{\ell-1}\}.
\end{align*}

We assume that $r$ be the dimension of $Q_r$. Then $r \le \ell -1$ and 
\begin{align*}
Q_r &=  \textup{span}\{\widetilde \psi_1, \widetilde\psi_2,\ldots, \widetilde\psi_{r}\} =	\textup{span}\{	\mathfrak{p}_h^1, 	\mathfrak{p}_h^2,\ldots, 	\mathfrak{p}_h^{\ell-1}\}.
\end{align*}
Therefore, for any $j=1,2,\ldots, \ell-1$ we have 
\begin{align*}
\mathfrak{u}_h^j = \sum_{i=1}^{\ell-1}\left( \mathfrak{u}_h^j, {\widetilde\varphi}_{i}\right)_{V} {\widetilde\varphi}_{i},\quad \textup{and}\quad 
\mathfrak{p}_h^j = \sum_{i=1}^{r}\left( \mathfrak{p}_h^j, {\widetilde\psi}_{i}\right) {\widetilde\psi}_{i}.
\end{align*}

For $i=1,2,\ldots, \ell$, we define  the sequences $\{\alpha_i^n\}_{n=1}^{N_T}$ by
\begin{align}\label{def_alpha2}
\begin{split}
\partial_t^+\alpha_i^{n}  +  \dfrac 1 {\mu_{m_i}} \alpha_i^{n} &=  c_i, \qquad n\ge 1,\\
\alpha_i^{0} & = 0.
\end{split}
\end{align}

\begin{lemma}\label{NonChang}
	If  Assumption \ref{Assumption1} holds, then the unique solution of \eqref{CN_Stoeks} is given by
	\begin{align}\label{Exact_Solution}
	\left[\begin{array}{cc}
	u_h^n\\[0.2cm]
	p_h^n
	\end{array}\right] = \sum_{i=1}^\ell\alpha_i^n\left[\begin{array}{cc}
	\phi_{m_i,h}\\[0.4cm]
	\chi_{m_i,h}
	\end{array}\right],\qquad  n=1,2,\ldots, N_T.
	\end{align} 
\end{lemma}	
\begin{proof}
	We only need  to check that \eqref{Exact_Solution} satisfies \eqref{CN_Stoeks}. Substitute  	\eqref{Exact_Solution} into \eqref{CN_Stoeks} we have 
	\begin{align*}
	\left(\partial_t^+ u_h^n, v_{h}\right)+  a\left( u_h^n,  v_{h}\right) + b\left( v_h,p_h^n\right) = \left( \sum_{i=1}^\ell (\partial_t^+\alpha_i^n)\phi_{m_i,h} , v_{h}\right) +   \left(\sum_{i=1}^\ell \frac{\alpha_i^n}{\mu_{m_i}}\phi_{m_i,h},  v_{h}\right),
	\end{align*}
	where we used the fact that $\phi_{m_i,h}$ is the eigenvector of \eqref{Weak_Stokes_Eigenvalue} corresponding to the eigenvalue $1/\lambda_{m_i}$ in the last equality. Therefore, by Assumption \ref{Assumption1} we have 
	\begin{align*}
	\left(\partial_t^+ u_h^n, v_{h}\right)+  a\left( u_h^n,  v_{h}\right) + b\left( v_h,p_h^n\right) = 	\left(\sum_{i=1}^\ell c_i \phi_{m_i,h}, v_{h}\right) = (\Pi f, v_h) = (f, v_h).
	\end{align*}
	Finally, it is easy to check that 
	\begin{align*}
	b\left( u_{h}^{n}, q_{h}\right) = b\left( \sum_{i=1}^\ell\alpha_i^n\phi_{m_i,h}, q_{h}\right)  = \sum_{i=1}^\ell\alpha_i^n b\left( \phi_{m_i,h}, q_{h}\right)=0,
	\end{align*}
	where we use $b\left( \phi_{m_i,h}, q_{h}\right)=0$ from \eqref{Weak_Stokes_Eigenvalue2}. This completes the proof.
\end{proof}

Due to the assumption \eqref{eig_ass}, it is easy to show that for all $n=1,2,\ldots, N_T$,
\begin{align}\label{c1c2}
\alpha_1^n c_2 = \alpha_2^n c_1.
\end{align}

\begin{lemma}\label{NonChang22}
	Let  $ (u_h^n, p_h^n)$ be the solution of \eqref{CN_Stoeks} and set $\textup{tol} = 0$ in \Cref{algorithm2}.  If Assumption \ref{Assumption1} and \eqref{eig_ass} hold, then for $n=1,2,\ldots, N_T$ we have
	\begin{align*}
	u_h^n \in V_r=	\textup{span}\{	\mathfrak{u}_h^1, 	\mathfrak{u}_h^2,\ldots, 	\mathfrak{u}_h^{\ell-1}\},\\
	p_h^n \in Q_r=	\textup{span}\{	\mathfrak{p}_h^1, 	\mathfrak{p}_h^2,\ldots, 	\mathfrak{p}_h^{\ell-1}\}.
	\end{align*}
\end{lemma}

\begin{proof}
	We rewrite the system \eqref{Vandemonde} as
	\begin{align}\label{Vandemonde1}
	\begin{split}
	\left[\begin{array}{cc}
	(\mathfrak{u}_h^1,\mathfrak{p}_h^1)  \\[0.2cm]
	(\mathfrak{u}_h^2,\mathfrak{p}_h^2)\\
	\vdots\\
	(\mathfrak{u}_h^{\ell-1},\mathfrak{p}_h^{\ell-1})
	\end{array}\right] &= \left[\begin{array}{cccc}
	c_2 \mu_{m_2}&c_3 \mu_{m_3}&\cdots&c_\ell\mu_{m_\ell}\\[0.2cm]
	c_2 \mu_{m_2}^2&c_3 \mu_{m_3}^2&\cdots&c_\ell\mu_{m_\ell}^2\\
	\vdots&\vdots&\cdots&\vdots\\
	c_2 \mu_{m_2}^{\ell-1}&c_3 \mu_{m_3}^{\ell-1}&\cdots&c_\ell\mu_{m_\ell}^{\ell-1}
	\end{array}\right] 
	\left[\begin{array}{cc}
	(\phi_{m_2,h},\chi_{m_2,h})  \\[0.2cm]
	(\phi_{m_3,h},\chi_{m_3,h}) \\
	\vdots\\
	(\phi_{m_\ell,h},\chi_{m_\ell,h}) 
	\end{array}\right]\\
	&+ \left[\begin{array}{cc}
	c_1 \mu_{m_1} \\[0.2cm]
	c_1 \mu_{m_1}^2\\
	\vdots\\
	c_1\mu_{m_1,h}^{\ell-1}
	\end{array}\right]	(\phi_{m_1,h},\chi_{m_1,h}).	
	\end{split}
	\end{align}	
	
	We denote the coefficient matrix of \eqref{Vandemonde1} by $S$, it is obvious that $S$ is inverterable since $\mu_{m_2}$, $\mu_{m_3}$, $\ldots$,  $\mu_{m_\ell}$ are distinct. Furthermore, since $\mu_{m_1} = \mu_{m_2}$, then 
	\begin{align*}
	S\left[\begin{array}{cc}
	1 \\[0.2cm]
	0\\
	\vdots\\
	0
	\end{array}\right] = \left[\begin{array}{cc}
	c_2 \mu_{m_2} \\[0.2cm]
	c_2 \mu_{m_2}^2\\
	\vdots\\
	c_2\mu_{m_2,h}^{r-1}
	\end{array}\right]= \left[\begin{array}{cc}
	c_2 \mu_{m_1} \\[0.2cm]
	c_2 \mu_{m_1}^2\\
	\vdots\\
	c_2\mu_{m_1,h}^{r-1}
	\end{array}\right].
	\end{align*} 
	This implies
	\begin{align*}
	\left[\begin{array}{cc}
	(\phi_{m_2,h},\chi_{m_2,h})  \\[0.2cm]
	(\phi_{m_3,h},\chi_{m_3,h}) \\
	\vdots\\
	(\phi_{m_\ell,h},\chi_{m_\ell,h}) 
	\end{array}\right]= S^{-1}	\left[\begin{array}{cc}
	(\mathfrak{u}_h^1,\mathfrak{p}_h^1)  \\[0.2cm]
	(\mathfrak{u}_h^2,\mathfrak{p}_h^2)\\
	\vdots\\
	(\mathfrak{u}_h^{\ell-1},\mathfrak{p}_h^{\ell-1})
	\end{array}\right]   - \dfrac{c_1}{c_2}\left[\begin{array}{cc}
	1 \\[0.2cm]
	0\\
	\vdots\\
	0
	\end{array}\right]	(\phi_{m_1,h},\chi_{m_1,h}).	
	\end{align*}
	Then by \Cref{NonChang} and \eqref{c1c2} we have 
	\begin{align*}
	\left[\begin{array}{cc}
	u_h^n\\[0.2cm]
	p_h^n
	\end{array}\right]  &= \sum_{i=1}^\ell\alpha_i^n
	\left[\begin{array}{cc}
	\phi_{m_i,h}\\[0.4cm]
	\chi_{m_i,h}
	\end{array}\right]= \alpha_1^n
	\left[\begin{array}{cc}
	\phi_{m_1,h}\\[0.4cm]
	\chi_{m_1,h}
	\end{array}\right] + \sum_{i=2}^\ell\alpha_i^n
	\left[\begin{array}{cc}
	\phi_{m_i,h}\\[0.4cm]
	\chi_{m_i,h}
	\end{array}\right]\\
	& = \alpha_1^n
	\left[\begin{array}{cc}
	\phi_{m_1,h}\\[0.4cm]
	\chi_{m_1,h}
	\end{array}\right]  + \sum_{i=2}^\ell\alpha_i^n \sum_{j=1}^{\ell-1} S^{-1}_{i-1,j}	\left[\begin{array}{cc}
	\mathfrak{u}_h^j\\[0.4cm]
	\mathfrak{p}_h^j
	\end{array}\right] - \dfrac{c_1}{c_2}\alpha_2^n	\left[\begin{array}{cc}
	\phi_{m_1,h}\\[0.4cm]
	\chi_{m_1,h}
	\end{array}\right]\\
	& =  \sum_{i=2}^\ell\alpha_i^n \sum_{j=1}^{\ell-1} S^{-1}_{i-1,j}	\left[\begin{array}{cc}
	\mathfrak{u}_h^j\\[0.4cm]
	\mathfrak{p}_h^j
	\end{array}\right]. 
	\end{align*}
	This completes the proof.
	
\end{proof}

\begin{proof}[Proof of \Cref{lemma2}]
	First, we take  $v_h\in V_r\subset  V_{h}^{\textup{div}}$ in \eqref{CN_Stoeks} to obtain
	\begin{align}\label{full_model_e}
	\left(\partial_t^+ u_h, v_{h}\right)+  a\left(u_h^n,  v_{h}\right) =\left(\sum_{i=1}^\ell c_i\phi_{m_i,h}, v_{h}\right). 	
	\end{align}
	Subtract \eqref{full_model_e} from \eqref{Velocity_onlyROM} and we let $e^n = u_r^n - u_h^n$, then  
	\begin{align*}
	\left(\partial_t^+ e^n, v_{h}\right)+  a\left(e^n,  v_{h}\right)  =0.
	\end{align*}
	By \Cref{NonChang22}  we  take $v_h = e^n\in V_r$ and the identity
	\begin{align*}
	(a-b, a) & = \dfrac 1 2 (\|a\|^2 - \|b^2\|) + \dfrac 1 2 \|a-b\|^2,\\
	\frac{1}{2}(3 a-4 b+c,  a )&= 
	\frac{1}{4}\left[\|a\|^{2}+\|2 a-b\|^{2}-\|b\|^{2}-\|2 b-c\|^{2}\right]+\frac{1}{4}\|a-2 b+c\|^{2},
	\end{align*}
	to get
	\begin{align*}
	\|e^1\|^2 - \|e^0\|^2 + \left\|e^1-e^0\right\|^2 +2 \Delta t\left\| e^1 \right\|_V^2 = 0.
	\end{align*}
	Since $e^0 = 0$, then  $e^1=0$. In other words
	\begin{align*}
	u_r^1  = u_h^1.
	\end{align*}
	For $n\ge 2$ we have
	\begin{align*}
	\left[\|e^n\|^{2} - 		 \|e^{n-1}\|^{2}\right] +  \left[\|2e^n-e^{n-1}\|^{2} - 		 \|2e^{n-1}-e^{n-2}\|^{2}\right] + \|e^n-2 e^{n-1}+e^{n-2}\|^{2} + 4\Delta t \|e^n\|_V  = 0.
	\end{align*}
	Summing both sides of the above identity from $n=2$ to $n=N_T$ completes the proof of \Cref{lemma2}.
\end{proof}

%
%
%
%
%
%

\subsection{Error estimate of the pressure}
Next, we present an error  analysis for the pressure. First,  the spaces $\mathcal S_h$ and $Q_r$ satisfy the  following inf-sup stability condition.
\begin{lemma}\textup{\cite[Proposition 2]{MR3341250}}\label{infsup2}
	Let $\beta_{h}>0$ be the inf-sup constant for the finite element basis in \eqref{Discrtr_inf_Sup}. The spaces $\mathcal S_h$ and $Q_r$ will then be inf-sup stable with a constant $\beta_r \geq \beta_{h}$, i.e.,
	$$
	\beta_{r}=	 	\inf _{p_r \in Q_r \backslash\{0\}} \sup _{s_{h} \in \mathcal S_{h} \backslash\{0\}} \frac{b\left(s_{h}, p_r\right)}{\left\| s_{h}\right\|_V\|p_r\|_Q}	 \geq \beta_{h}.
	$$
\end{lemma}

Now, we state the main result in this section.
\begin{tcolorbox}
	\begin{theorem}\label{error_for_p}
		Let $(u,p)$ be the solution of \eqref{Weak_Stokes} and $p_r^n$  be the solution of \eqref{algorithm2} by setting $\textup{tol} = 0$ in \Cref{algorithm1}.  If Assumption \ref{Assumption1} and Assumption \ref{Assumption2} hold, then we have 
		\begin{align*}
		\left\|p(t_n) - p_r^n\right\|  \le C\left( h^{k+1} + (\Delta t)^2 \right).
		\end{align*}
	\end{theorem}			
\end{tcolorbox}

\begin{proof}[Proof of \Cref{error_for_p}]
	First, we take $s_h\in \mathcal S_h$ in \eqref{CN_Stoeks} to get
	\begin{align*}
	b\left( s_{h}, p_{h}^{n}\right) &=\left(f, s_{h}\right)	-\left(\partial_t^+ u_h^n, s_{h}\right)-  a\left(u_h^n,  s_{h}\right).
	\end{align*}
	By \Cref{lemma2} we know that $u_h^n = u_r^n$ for all $n=1,2,\ldots, N_T$, then 
	\begin{align*}
	b\left( s_{h}, p_{h}^{n}\right) =\left(f, s_{h}\right)	-\left(\partial_t^+ u_r^n, s_{h}\right)-  a\left(u_r^n,  s_{h}\right).
	\end{align*}
	Due to fact that $u_r^n\in V_r\subset  V_{h}^{\textup{div}} $ and $s_h \in S_h \subset \left(V_{h}^{\textup{div}}\right)^{\perp}$, we have $ a\left(u_r^n,  s_{h}\right) = 0 $. Therefore,
	\begin{align}\label{auxi_sh1}
	b\left( s_{h}, p_{h}^{n}\right) =\left(f, s_{h}\right)	-\left(\partial_t^+ u_r^n, s_{h}\right).
	\end{align}
	Subtracting \eqref{auxi_sh1} from \eqref{pressure_r3} implies for all $n=1,2\ldots, N_T$ we have 
	\begin{align*}
	b(s_h, p_r^n-p_h^n ) =0.
	\end{align*}
	Recalling from \Cref{infsup2} that $\mathcal S_h$ and $Q_{r}$ are inf-sup stable with constant $\beta_r$, then 
	\begin{align}\label{presssureequa}
	p_r^n=p_h^n.
	\end{align}
	As a consequence, \Cref{lemma1}   and \eqref{presssureequa} give the proof of \Cref{error_for_p}. 
\end{proof}	

\section{General data}\label{NEwway}
In this section, we extend \Cref{algorithm0} to  general data. If the source term $f$ can be expressed or approximated  by a few only time dependent functions $f_{i}(t)$ and space dependent functions $g_i(x)$, i.e., 
\begin{align*}
f(t,x) = \sum_{i=1}^m f_i(t) g_i(x),
\end{align*}
or
\begin{align*}
f(t,x) \approx  \sum_{i=1}^m f(t_i^*,x)L_{m,i}(t):= \sum_{i=1}^m f_i(t) g_i(x),
\end{align*}
where $t_i^*$ are the $m$ Chebyshev interpolation nodes and $L_{m,i}(t)$ are the Lagrange interpolation functions:
\begin{align*}
t_{i}^*&=\frac{T}{2}+\frac{T}{2} \cos \frac{(2 i-1) \pi}{2 m}\quad \textup{for}\quad i=1,2,\ldots, m,\\
L_{m,i}(t) &= \frac{(t-t_1^*)\cdots(t-t_{i-1}^*)(t-t_{i+1}^*)\cdots (t-t_{m}^*)}{(t_i^*-t_1^*)\cdots(t_i^*-t_{i-1}^*)(t_i^*-t_{i+1}^*)\cdots (t_i^*-t_{m}^*)}.	
\end{align*}
Let $\{\varphi_i\}_{i=1}^N$ be the finite element basis function of $V_h$ and we then define the following vectors:
\begin{align}\label{def_F}
b_0 = [(u_0, \varphi_j)]_{j=1}^N,  \qquad
b_i = [(g_i, \varphi_j)]_{j=1}^N, \;\; i=1,2\ldots, m,\qquad b = [b_0\mid b_1\mid b_2\mid,\ldots, \mid b_m].
\end{align}

Now we can use \Cref{algorithm0}, the only difference is that at each step, the right hand side is not a vector, but a matrix. In some scenarios, the data is not continuous, such as the optimal control problem, we recommend to use the incremental SVD to compress the data first and then apply the ROM; see \cite{Zhang2022isvd,MR3775096,WalingtonWeberZhang} for more details.

Next, we present several numerical tests to show the accuracy and efficiency of our ROM. We let  $\Omega = (0,1)^2$,  the final time $T=1$, the initial condition $u_0 = 0$ and the body force
\begin{align*}
f = [f_1, f_2]^\top,\quad f_1 = \sin(tx),\quad f_2 = \cos(tx).
\end{align*}
Since the exact solution is not known, then we compute the error between the ROM and the $P_2-P_1$ Taylor-Hood  (TH) method. For both methods, we use BDF2 for the time discretization and take  time step $\Delta t = h^{3/2}$ and $h$ is the mesh size. For the ROM, we choose $\ell=5, m = 8, \texttt{tol}=10^{-15}$.  We report the error at the final time $T=1$ and the wall  time (WT) in \Cref{table_3}. We see that the convergence rate of the ROM is the same as the standard TH-method. 
\begin{table}[H]
	\centering
	{
		\begin{tabular}{c|c|c|c|c|c|c}
			\cline{1-7}
			
			$h$	& $1/2^2$ &$1/2^3$ 
			&$1/2^4$ 
			&$1/2^5$  &$1/2^6$ &$1/2^7$ 
			\\
			\cline{1-7}
			{WT of TH}
			& 	  0.15&   0.10&   0.9&   13.1&   218&   3844\\ 
			\cline{1-7}
			{WT of ROM}
			& 	  0.17&   0.05&   0.16&   0.75&   4.08&   23.3\\ 
			\cline{1-7}
			{$\mathcal E_u$}
			& 4.10E-10&	  4.39E-10&   5.53E-10&   1.24E-10&   1.29E-10&  5.58E-10  \\ 
			\Xhline{0.1pt}
			
			{$\mathcal E_p$}
			& 3.20E-07&	  3.26E-07&   2.96E-07&   2.93E-07&   2.93E-07&  2.93E-07  \\ 
			\Xhline{0.1pt}
		\end{tabular}
	}
	\caption{The dimension and wall time (seconds) of the  ROM. The $L^2$-norm error between the solutions of the FEM and the ROM at the final time $T=1$.}\label{table_3}
\end{table}

\section{Conclusion} In the paper, we followed the idea in \cite{WalingtonWeberZhang} and proposed a new reduced order model (ROM) to  imcompressible Stokes equations. We showed that the eigenvalues of the velocity data  are exponential decay. Furthermore, the dimension of the reduced pressure space is determined by the reduced velocity subspace. Under some assumptions, we proved that the solutions of the ROM and the FEM are the same.  There are many interesting directions for the future research. First, we see the error of the pressure is much larger than the error of the velocity, this suggests  us to apply pressure-robust algorithm to generate the sequence; see \cite{chen2022new} for more details. Second, we will explore the Stokes-Darcy equation and related optimal control problems; see \cite{MR4381532,MR4169689}.

\bibliographystyle{siamplain}
\bibliography{Stokes}

\end{document}